\documentclass[a4paper,twopage,reqno,12pt]{amsart}
\usepackage[top=30mm,right=30mm,bottom=30mm,left=30mm]{geometry}

\usepackage{graphicx}
\usepackage{amsmath}
\usepackage{amssymb}
\usepackage{amsfonts}
\usepackage{amsthm}
\usepackage{amstext}
\usepackage{amsbsy}
\usepackage{amsopn}
\usepackage{amscd}
\usepackage{enumerate}
\usepackage{color}
\usepackage[colorlinks]{hyperref}
\usepackage{multirow}
\usepackage{lipsum}
\usepackage{rotating}
\usepackage{float}
\usepackage{lscape}
\usepackage[]{algorithm2e}
\usepackage{tabularx}

\newtheorem{theorem}{Theorem}[section]
\newtheorem{corollary}[theorem]{Corollary}
\newtheorem{lemma}[theorem]{Lemma}
\newtheorem{proposition}[theorem]{Proposition}

\theoremstyle{definition}

\newtheorem{remark}[theorem]{Remark}
\newtheorem{example}[theorem]{Example}

\numberwithin{equation}{section}

\newcommand{\Alt}{Alt}
\newcommand{\Sym}{Sym}

\newcommand{\Aut}{\mathrm{Aut}}

\newcommand{\Out}{\mathrm{Out}}

\newcommand{\Dmc}{\mathcal{D}}
\newcommand{\Bmc}{\mathcal{B}}

\newcommand{\Pmc}{\mathcal{P}}

\newcommand{\e}{\epsilon}

\renewcommand{\leq}{\leqslant}
\renewcommand{\geq}{\geqslant}

\newcommand{\imod}[1]{\allowbreak\mkern4mu({\operator@font mod}\,\,#1)}

\begin{document}
 \title[]{Flag-transitive block designs and finite exceptional simple groups of Lie type}

 \author[S.H. Alavi]{Seyed Hassan Alavi}%
 \address{Seyed Hassan Alavi, Department of Mathematics, Faculty of Science, Bu-Ali Sina University, Hamedan, Iran.
 }%
 \email{alavi.s.hassan@basu.ac.ir and  alavi.s.hassan@gmail.com (G-mail is preferred)}

 \date{\today}%

\begin{abstract}
In this article, we study $2$-designs with $\gcd(r,\lambda)=1$ admitting a flag-transitive almost simple automorphism group with socle a finite simple exceptional group of Lie type. We obtain four infinite families of such designs and provide some examples in each of these families.
\end{abstract}

\maketitle

\section{Introduction}\label{sec:intro}

A $2$-design $\Dmc$ with parameters $(v,k,\lambda)$ is a pair $(\Pmc,\Bmc)$ with a set $\Pmc$ of $v$ points and a set $\Bmc$ of $b$ blocks such that each block is a $k$-subset of $\Pmc$ and each two distinct points are contained in $\lambda$ blocks. The \emph{replication number} $r$ of $\Dmc$ is the number of blocks incident with a given point. A \emph{symmetric} design is a $2$-design with the same number of points and blocks, that is to say, $v=b$, or equivalently, $r=k$. A \emph{flag} of $\Dmc$ is a point-block pair $(\alpha, B)$ such that $\alpha \in B$. An \emph{automorphism} of $\Dmc$ is a permutation on $\Pmc$ which maps blocks to blocks and preserving the incidence. The \emph{full automorphism} group $\Aut(\Dmc)$ of $\Dmc$ is the group consisting of all automorphisms of $\Dmc$. For $G\leq \Aut(\Dmc)$, $G$ is called \emph{flag-transitive} if $G$ acts transitively on the set of flags. The group $G$ is said to be \emph{point-primitive} if $G$ acts primitively on $\Pmc$. A group $G$ is called \emph{almost simple} with socle $X$ if $X\unlhd G\leq \Aut(X)$, where $X$ is a nonabelian simple group. We here adopt the standard Lie notation for groups of Lie type. Further definitions and notation can be found in Subsection~\ref{sec:defn} below.

The main aim of this paper is to study $2$-designs with flag-transitive automorphism groups. In 1988, Zieschang \cite{a:ZIESCHANG} proved that if an automorphism group $G$ of a $2$-design with $\gcd(r,\lambda)=1$ is flag-transitive, then $G$ is a point-primitive group of almost simple or affine type. Such designs admitting an almost simple automorphism group with socle being an alternating group, a classical simple group or a sporadic simple group have been studied in \cite{a:ABD-Un-CP,a:ADM-All,a:Zhou-lam-large-sporadic,a:Zhan-nonsym-sprodic,a:Zhou-nonsym-alt,a:Zhu-sym-alternating}.
This problem for symmetric designs of affine type automorphism groups and almost simple groups with socle a finite simple exceptional group have also been  treated in \cite{a:ABD-Exp,a:Biliotti-affine-17}.
This paper is devoted to investigating the  $2$-designs with $\gcd(r,\lambda)=1$ and flag-transitive almost simple automorphism groups whose socle is a finite simple exceptional group of Lie type. Here we extend our result \cite[Corolary 1.3]{a:ABD-Exp}, and prove that there exist only four families of such $2$-designs:

\begin{theorem}\label{thm:main}
    Let $\Dmc=(\Pmc,\Bmc)$ be a non-trivial $(v, k, \lambda)$ design with $\lambda\geq 1$, and let $G$ be a flag-transitive automorphism group of $\Dmc$ whose socle $X$ is a finite simple exceptional group of Lie type. Let also $H=G_{\alpha}$ and $K=G_{B}$, for a flag $(\alpha,B)$ of $\Dmc$. If the replication number $r$ of $\Dmc$ is coprime to $\lambda$, then  $H$ is a parabolic subgroup of $G$,
    one of the following holds:
    \begin{enumerate}[{\rm \quad (a)}]
        \item $X={}^2\!B_{2}(q)$, $H\cap X\cong q^{2}{:}(q-1)$ and  $K\cap X \cong q{:}(q-1)$ with $q=2^a$ and $a\geq 3$ odd, and $\Dmc$ is a block design with parameters $v=q^2+1$, $b=q(q^{2}+1)$, $r=q^{2}$,  $k=q$ and $\lambda=q-1$;
        \item  $X={}^2\!G_{2}(q)$ and $H\cap X\cong q^3{:}(q-1)$, and $K\cap X\cong 2\times A_{1}(q)$ with $q=3^a\geq 27$, and $\Dmc$ is the Ree Unital space $U_{R}(q)$ with parameters $v=q^{3}+1$, $b=q^{2}(q^{2}-q+1)$, $r=q^{2}$, $k=q+1$ and $\lambda=1$;
        \item $X={}^2\!G_{2}(q)$, $H\cap X\cong q^3{:}(q-1)$ and $K\cap X\cong q{:}(q-1)$ with $q=3^a$ and $a\geq 3$ odd, and $\Dmc$ is a block design with parameters $v=q^3+1$, $b=q^{2}(q^{3}+1)$, $r=q^{3}$,  $k=q$ and $\lambda=q-1$;
        \item $X={}^2\!G_{2}(q)$, $H\cap X\cong q^3{:}(q-1)$ and $K\cap X\cong q^2{:}(q-1)$ with $q=3^a$ and $a\geq 3$ odd, and $\Dmc$ is a block design with parameters $v=q^3+1$, $b=q(q^{3}+1)$, $r=q^{3}$,  $k=q^{2}$ and $\lambda=q^{2}-1$.
    \end{enumerate}
\end{theorem}


In Section~\ref{sec:examples} below, in addition to making some general discussion on existence of the  designs obtained in Theorem \ref{thm:main}, we also provide some examples for small values of $q$. In order to prove Theorem~\ref{thm:main} in Section~\ref{sec:proof},  we first observe that the group $G$ is point-primitive, and so the point-stabiliser $H$ is maximal in $G$. In particular,  flag-transitivity of $G$ implies that $H$ is large, that is to say, $|G|\leq |H|^{3}$, and then we can apply  \cite[Theorem~1.6]{a:ABD-Exp} in which the large maximal subgroups of almost simple groups whose socle $X$ is a finite simple exceptional group of Lie type are determined. We then analyse all possible cases and prove that the only possible designs are those given in Theorem \ref{thm:main}.

\subsection{Definitions and notation}\label{sec:defn}

All groups and incidence structures in this paper are finite.
We here write $\Alt_{n}$ and $\Sym_{n}$ for the alternating group and the symmetric group on $n$ letters, respectively, and we denote by ``$n$'' the cyclic group of order $n$. Recall  that we in this paper adopt the standard Lie notation for groups of Lie type. For example, we write $A_{n-1}(q)$ and $A_{n-1}^{-}(q)$ in place of $PSL_{n}(q)$ and $PSU_{n}(q)$, respectively, $D_n^{-}(q)$ instead of  $P\Omega_{2n}^{-}(q)$, and $E_6^{-}(q)$ for ${}^2\!E_6(q)$.
We may also assume $q>2$ when $G=G_2(q)$ since $G_{2}(2)$ is not simple and $G_2(2)' \cong A^{-}_{2}(3)$. Moreover, we view the Tits group ${}^2\!F_4(2)'$ as a sporadic group. For a given positive integer $n$ and a prime divisor $p$ of $n$, we denote the $p$-part of $n$ by $n_{p}$, that is to say, $n_{p}=p^{t}$ if $p^{t}\mid n$ but $p^{t+1}\nmid n$.
Further notation and definitions in both design theory and group theory are standard and can be found, for example, in \cite{b:Atlas,b:Dixon,b:KL-90,b:lander}.

\section{Examples and comments}\label{sec:examples}

In this section, we provide some examples of $2$-designs with $\gcd(r,\lambda)=1$ admitting a flag-transitive almost simple automorphism group with socle $X$.
We note that all groups $X$ in Theorem \ref{thm:main}(a)-(d) are $2$-transitive in their coset actions of the maximal parabolic subgroups $H\cap X$ in $X$, and so by \cite[2.2.8]{b:dembowski}, the  groups $X$ are also flag-transitive.

\begin{example}[Suzuki designs]\label{ex:suzuki}
Suppose that $G$ is an almost simple group with socle $X={}^2\!B_{2}(q)$ for $q=2^a$ and $a\geq 3$ odd. Let $H$ and  $K$ be subgroups of $G$ such that $H\cap X\cong q^{2}{:}(q-1)$ and $K\cap X \cong q{:}(q-1)$ as in Theorem~\ref{thm:main}(a). The coset geometry $(X,H\cap X,K\cap X)$ may give rise to a $2$-design with parameters  $v=q^2+1$, $b=q(q^{2}+1)$, $r=q^{2}$,  $k=q$ and $\lambda=q-1$. By \cite[2.3.8]{b:dembowski}, $X$ is flag-transitive. If $\Pmc=\{1,\ldots,v\}$ and  $B$ is an orbit of $K\cap X$ of length $k=q$, then by \cite[Proposition~4.6]{b:Beth-I-99}, $(\Pmc,B^{X})$ is a $2$-design  with parameters $(q^2+1,q,q-1)$ which is not symmetric.
For $q\in \{8,32\}$, we construct this type of designs with explicit base blocks in Table~1.
\end{example}

\begin{example}[Ree unitals]\label{ex:ree-unital}
The Ree Unital spaces $U_{R}(q)$ are first discovered by L\"{u}neburg \cite{a:Luneburg-66}, and these examples arose from studying flag-transitive linear spaces \cite{a:Kantor-85,a:Kleidman-87}. This disign has parameters  $(q^{3}+1,q+1,1)$ with $q=3^a\geq 27$.   The points and blocks of $U_{R}(q)$ are the Sylow $3$-subgroups and the involutions of $^{2}\!G_{2}(q)$, respectively, and a point is incident with a block if the  block normalizes the  point. This incidence structure is a linear space and any group with $^{2}\!G_{2}(q)\leq G\leq \Aut(^{2}\!G_{2}(q))$ acts flag-transitively. This design is not symmetric. Note for $q=3$ that the Ree Unital $U_{R}(3)$ is isomorphic to the  Witt-Bose-Shrikhande space $W(8)$ as $^{2}\!G_{2}(3)'$ is isomorphic to $A_{1}(8)$, see \cite{a:delan-linear-space,a:Kleidman-87,a:Saxl2002}.
\end{example}

\begin{example}[Ree designs]\label{ex:ree-designs}
Suppose that $G$ is an almost simple group with socle $X={}^2\!G_{2}(q)$ for $q=2^a$ and $a\geq 3$ odd. Let $H$, $K_{1}$ and $K_{2}$ be subgroups of $G$ such that $H\cap X\cong q^{3}{:}(q-1)$, $K_{1}\cap X=q{:}(q-1)$ and $K_{2}\cap X \cong q^{2}{:}(q-1)$ as in Theorem~\ref{thm:main}(c)-(d). The coset geometries $(X,H\cap X,K_{i}\cap X)$ may give rise to the $2$-designs with parameters  $v=q^3+1$, $b=q(q^{3}+1)$, $r=q^{3}$,  $k=q^{i}$ and $\lambda=q^{i}-1$, for $i=1,2$. Since $G$ is $2$-transitive on the points set of this structure and $\gcd(r,\lambda) = 1$,  $X$ is flag-transitive \cite[2.3.8]{b:dembowski}. Note that $H\cap K_{i} \cap X$ is a cyclic group of order $q-1$. Let $B_{i}$ be an orbit of $K_{i}\cap X$ of length $k=q^{i}$ with $i=1,2$. If $\Pmc=\{1,\ldots,v\}$, then since $X$ is $2$-transitive, \cite[Proposition~4.6]{b:Beth-I-99} gives rise to a $2$-design $\Dmc_{i}=(\Pmc,B_{i}^{X})$ with parameters $(q^3+1,q^{i},q^{i}-1)$, for $i=1,2$, which is not symmetric, and the group $G$ is flag-transitive on $\Dmc_i$. In Table 1, for $q=27$, we introduced base blocks for these type of designs.
\end{example}

In Table \ref{tbl:examples}, we give the base block $B$ of $\Dmc$ for the case where $G=X$ is $^{2}\!B_{2}(8)$, $^{2}\!B_{2}(32)$ or $^{2}\!G_{2}(27)$.
For each example in this table, we use the  computer software \textsf{GAP} \cite{GAP4} for computational arguments, and we also use the permutation representations of both groups and parabolic subgroups given in the web version of Atlas \url{ (http://brauer.maths.qmul.ac.uk/Atlas/v3/)}.
As an example, the  Suzuki group ${}^2\!B_{2}(8)$ has a unique conjugacy class of the parabolic subgroup $H:=X_{\alpha}=2^{3+3}{:}7$, the conjugate subgroup $M$ of $H$ has one subgroup $K:=X_{B}$ of order $56$ (up to conjugation), and this subgroup has exactly one orbit $B=\{1, 4, 5, 10, 32, 48, 52, 55 \}$ of length $8$, and so for $\Pmc=\{1,\ldots,65\}$, the incidence structure  $(\Pmc,B^{G})$ is a $2$-design with parameters $(65,520,64,8,7)$. We note here that for the case where $X=\ ^{2}\!G_{2}(27)$ in Theorem \ref{thm:main}(c) and the designs with possible parameters $(19684,551853,19683,27,26)$, we obtain two base blocks in Table \ref{tbl:examples} (Nr. 3). These base blocks are orbits of two non-conjugate subgroups $K$ in $M$, and at this stage we do not know if these two designs are isomorphic.

\begin{table}[!htbp]
    \scriptsize
    \caption{\small Designs with flag-transitive and point-primitive automorphism groups $^{2}\!B_{2}(8)$, $^{2}\!B_{2}(32)$ and $^{2}\!G_{2}(27)$.}\label{tbl:examples}
    \resizebox{\textwidth}{!}{
        \begin{tabular}{llllllllll}
            \hline
            Nr.&
            $X$&
            $X_{\alpha}$&
            $X_{B}$&
            $v$&
            $b$&
            $r$&
            $k$ &
            $\lambda$ &
            Comments\\
            \hline
            $1$ &
            $^{2}\!B_{2}(8)$&
            $2^{3+3}{:}7$&
            $2^{3}{:}7$&
            $65$ &
            $520$ &
            $64$ &
            $8$ &
            $7$ &
            Theorem~\ref{thm:main}(a)\\
            \multicolumn{10}{p{15cm}}{Base block:} \\
            \multicolumn{10}{p{15cm}}{
                \tiny   $B=\{$ $1$, $4$, $5$, $10$, $32$, $48$, $52$, $55$ $\}$}\\
            \hline
            Nr.&
            $X$&
            $X_{\alpha}$&
            $X_{B}$&
            $v$&
            $b$&
            $r$&
            $k$ &
            $\lambda$ &
            Comments\\
            \hline
            $2$ &
            $^{2}\!B_{2}(32)$ &
            $2^{5+5}{:}31$&
            $2^{5}{:}31$&
            $1025$ &
            $32800$ &
            $1024$ &
            $32$ &
            $31$ &
            Theorem~\ref{thm:main}(a)\\
            \multicolumn{10}{p{15cm}}{Base block:} \\
            \multicolumn{10}{p{15cm}}{
                \tiny $B=\{$ $1$, $44$, $69$, $77$, $97$, $106$, $153$, $204$, $251$, $273$, $300$, $317$, $319$, $326$, $333$, $361$, $374$, $383$, $393$, $401$, $423$, $459$, $594$, $601$, $608$, $666$, $823$, $841$, $931$, $976$, $1015$,
                $1024$ $\}$}\\
            \hline
            Nr.&
            $X$&
            $X_{\alpha}$&
            $X_{B}$&
            $v$&
            $b$&
            $r$&
            $k$ &
            $\lambda$ &
            Comments\\
            \hline
            $3$ &
            $^{2}\!G_{2}(27)$&
            $3^{3+3+3}{:}26$ &
            $3^{3}{:}26$ &
            $19684$ &
            $551853$ &
            $19683$ &
            $27$ &
            $26$ &
            Theorem~\ref{thm:main}(c)\\
            \multicolumn{10}{p{14cm}}{Base block:} \\
            \multicolumn{10}{p{15cm}}{
                \tiny $B_{1}=\{$ $1$, $54$, $676$, $941$, $1426$, $1443$, $1464$, $1902$, $1910$, $3984$, $4612$, $4813$, $4884$, $5223$, $5849$, $8312$, $8423$, $8670$, $9369$, $11678$, $11765$, $13032$, $13580$, $14087$, $15253$, $
                16549$, $17856$ $\}$
            }\\
            \multicolumn{10}{p{15cm}}{
                \tiny $B_{2}=\{$ $1$, $42$, $95$, $1385$, $3667$, $3964$, $5173$, $5447$, $5974$, $6757$, $6772$, $7263$, $8170$, $9075$, $10131$, $10447$, $11119$, $11292$, $12766$, $12874$, $13493$, $13786$, $14163$, $16504$, $
                16618$, $17126$, $19415$ $\}$
            }\\
            \hline
            Nr.&
            $X$&
            $X_{\alpha}$&
            $X_{B}$&
            $v$&
            $b$&
            $r$&
            $k$ &
            $\lambda$ &
            Comments\\
            \hline
            $4$ &
            $^{2}\!G_{2}(27)$&
            $3^{3+3+3}{:}26$ &
            $3^{3+3}{:}26$ &
            $19684$ &
            $20439$ &
            $19683$ &
            $729$ &
            $728$ &
            Theorem~\ref{thm:main}(d)\\
            \multicolumn{10}{p{14cm}}{Base block:} \\
            \multicolumn{10}{p{15cm}}{\tiny $B=\{$ $1$, $42$, $51$, $54$, $58$, $63$, $86$, $95$, $102$, $124$, $142$, $153$, $195$, $371$, $409$, $420$, $438$, $471$, $540$, $575$, $606$, $610$, $676$, $704$, $711$, $734$, $809$, $819$, $832$, $857$, $885$, $934$, $935$, $941$, $1041$, $1042$, $1064$, $1108$, $1246$, $1305$, $1355$, $1357$, $1362$, $1366$, $1375$, $1377$, $1385$, $1388$, $1423$, $1426$, $1433$, $1442$, $1443$, $1464$, $1478$, $1509$, $1538$, $1547$, $1550$, $1729$, $1745$, $1758$, $1779$, $1794$, $1798$, $1818$, $1841$, $1849$, $1873$, $1902$, $1903$, $1910$, $1926$, $1965$, $1970$, $1995$, $2022$, $2074$, $2086$, $2089$, $2103$, $2121$, $2198$, $2207$, $2220$, $2224$, $2294$, $2316$, $2399$, $2407$, $2414$, $2415$, $2430$, $2439$, $2441$, $2471$, $2539$, $2550$, $2731$, $2749$, $2838$, $2846$, $2849$, $2892$, $2930$, $2962$, $2987$, $2994$, $3014$, $3038$, $3074$, $3104$, $3106$, $3123$, $3136$, $3161$, $3238$, $3258$, $3328$, $3333$, $3342$, $3363$, $3365$, $3373$, $3375$, $3376$, $3440$, $3492$, $3498$, $3508$, $3517$, $3555$, $3593$, $3606$, $3634$, $3667$, $3686$, $3755$, $3767$, $3769$, $3800$, $3819$, $3852$, $3875$, $3881$, $3891$, $3909$, $3913$, $3925$, $3942$, $3954$, $3964$, $3984$, $3993$, $4069$, $4102$, $4115$, $4123$, $4134$, $4151$, $4173$, $4223$, $4262$, $4271$, $4289$, $4350$, $4375$, $4378$, $4385$, $4404$, $4459$, $4539$, $4608$, $4612$, $4644$, $4654$, $4713$, $4750$, $4766$, $4779$, $4813$, $4884$, $4991$, $5017$, $5070$, $5097$, $5111$, $5128$, $5154$, $5172$, $5173$, $5218$, $5223$, $5275$, $5320$, $5346$, $5378$, $5398$, $5401$, $5447$, $5450$, $5468$, $5490$, $5500$, $5521$, $5592$, $5666$, $5713$, $5724$, $5742$, $5752$, $5788$, $5847$, $5849$, $5881$, $5895$, $5898$, $5901$, $5915$, $5932$, $5974$, $5979$, $6019$, $6029$, $6043$, $6051$, $6062$, $6109$, $6114$, $6124$, $6190$, $6194$, $6213$, $6220$, $6236$, $6264$, $6277$, $6306$, $6311$, $6448$, $6482$, $6497$, $6510$, $6559$, $6566$, $6626$, $6629$, $6666$, $6698$, $6704$, $6708$, $6734$, $6741$, $6757$, $6768$, $6772$, $6777$, $6862$, $6872$, $6873$, $6943$, $6944$, $6973$, $6976$, $6990$, $7030$, $7044$, $7108$, $7154$, $7163$, $7234$, $7255$, $7257$, $7263$, $7323$, $7327$, $7334$, $7338$, $7342$, $7395$, $7458$, $7472$, $7522$, $7544$, $7561$, $7572$, $7592$, $7604$, $7613$, $7629$, $7657$, $7672$, $7681$, $7687$, $7688$, $7758$, $7797$, $7799$, $7800$, $7806$, $7834$, $7839$, $7873$, $7894$, $8043$, $8049$, $8057$, $8085$, $8157$, $8158$, $8168$, $8170$, $8194$, $8197$, $8202$, $8253$, $8260$, $8276$, $8289$, $8293$, $8299$, $8312$, $8423$, $8481$, $8487$, $8488$, $8554$, $8593$, $8670$, $8687$, $8716$, $8754$, $8760$, $8782$, $8827$, $8865$, $8880$, $8884$, $8897$, $8951$, $8998$, $9001$, $9068$, $9075$, $9079$, $9139$, $9158$, $9197$, $9225$, $9253$, $9349$, $9355$, $9366$, $9369$, $9382$, $9427$, $9451$, $9471$, $9581$, $9611$, $9655$, $9677$, $9679$, $9750$, $9780$, $9872$, $9947$, $9991$, $10010$, $10078$, $10095$, $10113$, $10116$, $10121$, $10131$, $10155$, $10258$, $10285$, $10290$, $10351$, $10381$, $10402$, $10409$, $10447$, $10465$, $10473$, $10545$, $10632$, $10667$, $10703$, $10726$, $10751$, $10786$, $10870$, $10947$, $10967$, $10976$, $11000$, $11002$, $11044$, $11048$, $11053$, $11064$, $11107$, $11119$, $11177$, $11181$, $11231$, $11237$, $11252$, $11280$, $11292$, $11296$, $11357$, $11373$, $11404$, $11405$, $11420$, $11458$, $11500$, $11522$, $11528$, $11536$, $11569$, $11591$, $11647$, $11678$, $11728$, $11765$, $11825$, $11834$, $11839$, $11840$, $11841$, $11870$, $11902$, $11948$, $12019$, $12089$, $12110$, $12111$, $12157$, $12170$, $12194$, $12209$, $12256$, $12263$, $12268$, $12299$, $12301$, $12364$, $12393$, $12431$, $12434$, $12491$, $12554$, $12563$, $12586$, $12594$, $12639$, $12661$, $12692$, $12757$, $12766$, $12773$, $12798$, $12807$, $12815$, $12843$, $12874$, $12876$, $12900$, $12912$, $12913$, $12920$, $12921$, $12922$, $12923$, $12925$, $12951$, $13016$, $13032$, $13074$, $13088$, $13097$, $13131$, $13143$, $13166$, $13195$, $13217$, $13219$, $13230$, $13249$, $13250$, $13276$, $13293$, $13326$, $13329$, $13351$, $13362$, $13364$, $13368$, $13391$, $13454$, $13471$, $13485$, $13493$, $13499$, $13525$, $13580$, $13581$, $13603$, $13625$, $13786$, $13812$, $13817$, $13826$, $13917$, $13920$, $13929$, $13948$, $13967$, $14020$, $14067$, $14087$, $14138$, $14157$, $14163$, $14246$, $14247$, $14261$, $14279$, $14281$, $14303$, $14312$, $14339$, $14344$, $14494$, $14497$, $14546$, $14557$, $14624$, $14632$, $14651$, $14694$, $14695$, $14707$, $14711$, $14750$, $14759$, $14762$, $14777$, $14837$, $14894$, $14917$, $14931$, $14950$, $14962$, $15004$, $15071$, $15086$, $15095$, $15112$, $15131$, $15147$, $15150$, $15156$, $15182$, $15217$, $15253$, $15263$, $15301$, $15318$, $15369$, $15399$, $15427$, $15519$, $15559$, $15561$, $15568$, $15573$, $15609$, $15611$, $15628$, $15677$, $15685$, $15723$, $15746$, $15787$, $15793$, $15816$, $15819$, $15827$, $15833$, $15848$, $15944$, $15958$, $15993$, $16043$, $16060$, $16077$, $16100$, $16212$, $16215$, $16228$, $16334$, $16349$, $16359$, $16423$, $16472$, $16504$, $16522$, $16549$, $16568$, $16618$, $16619$, $16629$, $16643$, $16654$, $16682$, $16724$, $16725$, $16730$, $16768$, $16776$, $16821$, $16825$, $16842$, $16913$, $16965$, $16987$, $17004$, $17044$, $17055$, $17086$, $17093$, $17119$, $17126$, $17150$, $17180$, $17247$, $17302$, $17303$, $17364$, $17382$, $17437$, $17458$, $17480$, $17492$, $17511$, $17514$, $17533$, $17588$, $17628$, $17658$, $17692$, $17696$, $17737$, $17801$, $17856$, $17864$, $17917$, $17919$, $17925$, $17931$, $17943$, $17969$, $18134$, $18137$, $18166$, $18173$, $18178$, $18244$, $18264$, $18293$, $18351$, $18353$, $18354$, $18363$, $18407$, $18425$, $18434$, $18441$, $18476$, $18495$, $18523$, $18525$, $18594$, $18651$, $18660$, $18675$, $18730$, $18765$, $18777$, $18810$, $18836$, $18840$, $18866$, $18896$, $18906$, $18992$, $19047$, $19068$, $19100$, $19121$, $19125$, $19131$, $19171$, $19256$, $19261$, $19342$, $19415$, $19418$, $19450$, $19487$, $19504$, $19526$, $19533$, $19534$, $19588$, $19592$, $19593$, $19608$, $19613$, $19629$, $19646$, $19656$, $19678$, $19682$ $\}$}\\
            \hline
        \end{tabular}
    }
\end{table}

\section{Preliminaries}\label{sec:pre}

In this section, we state some useful facts in both design theory and group theory. Lemma \ref{lem:New} below is an elementary result on subgroups of almost simple groups.

\begin{lemma}\label{lem:New}{\rm \cite[Lemma 2.2]{a:ABD-PSL2}}
    Let $G$  be an almost simple group with socle $X$, and let $H$ be maximal in $G$ not containing $X$. Then $G=HX$ and
    $|H|$ divides $|\Out(X)|\cdot |H\cap X|$.
\end{lemma}

\begin{lemma}\label{lem:Tits}
    Suppose that $\Dmc$ is block design with parameters $(v,k,\lambda)$ admitting a flag-transitive and point-primitive almost simple automorphism group $G$ with socle $X$ of Lie type in characteristic $p$. Suppose also that the point-stabiliser $G_{\alpha}$, not containing $X$, is not a parabolic subgroup of $G$. Then $\gcd(p,v-1)=1$.
\end{lemma}
\begin{proof}
    Note that $G_{\alpha}$ is maximal in $G$, then by Tits' Lemma \cite[1.6]{a:tits}, $p$ divides $|G:G_{\alpha}|=v$, and so  $\gcd(p,v-1)=1$.
\end{proof}

If a group $G$ acts on a set $\Pmc$ and $\alpha\in \Pmc$, the \emph{subdegrees} of $G$ are the size of orbits of the action of the point-stabiliser $G_\alpha$ on $\Pmc$.

\begin{lemma}\label{lem:subdeg}{\rm \cite[3.9]{a:LSS1987}}
    If $X$ is a group of Lie type in characteristic $p$, acting on the set
    of cosets of a maximal parabolic subgroup, and $X$ is neither $A_{n-1}(q)$, $D_{n}(q)$
    (with $n$ odd), nor $E_{6}(q)$, then there is a unique subdegree which is a power of $p$.
\end{lemma}

\begin{remark}\label{rem:subdeg}
    We remark that even in the cases excluded in Lemma~\ref{lem:subdeg}, many of the maximal parabolic subgroups still have the property as asserted, see proof of  \cite[Lemma 2.6]{a:Saxl2002}. In particular, for an almost simple group $G$ with socle $X=E_{6}(q)$, if $G$  contains a
    graph automorphism or $H =P_{i}$ with $i$ one of $2$ and $4$, the conclusion of Lemma~\ref{lem:subdeg} is still true.
\end{remark}

\begin{proposition}{\rm \cite[2.3.7(a)]{b:dembowski} and \cite{a:ZIESCHANG}}\label{prop:flag}
    Let $\Dmc$ be a $2$-design whose replication number $r$ is coprime to $\lambda$. If $G$ is a flag-transitive automorphism group of $\Dmc$, then $G$ is a primitive group on points set and it is of almost simple or affine type.
\end{proposition}

\begin{lemma}\label{lem:six} {\rm \cite[Lemmas 5 and 6]{a:Zhou-nonsym-alt}}
    Let $\Dmc$ be a $2$-design with prime replication number $r$, and let $G$ be a flag-transitive automorphism group of $\Dmc$. If $\alpha$ is a point in $\Pmc$ and $H:=G_{\alpha}$, then
    \begin{enumerate}[\rm \quad (a)]
        \item $r(k-1)=\lambda(v-1)$.  In particular, if $r$ is coprime to $\lambda$, then $r$ divides $v-1$ and $\gcd(r,v)=1$;
        \item $vr=bk$;
        \item $r\mid |H|$ and $\lambda v<r^2$;
        \item $r\mid d$, for all nontrivial subdegrees $d$ of $G$.
    \end{enumerate}
\end{lemma}

For a point-stabiliser $H$ of an automorphism group $G$ of a flag-transitive design $\Dmc$, by Lemma~\ref{lem:six}(c), we conclude that $\lambda|G|\leq |H|^{3}$, and so by \cite[Theorem 1.6]{a:ABD-Exp}\label{thm:large-ex} we have the list of possible point-stabiliser subgroups  $H$ of $G$.

\begin{corollary}\label{cor:large}
    Let $\Dmc$ be a symmetric design admitting a point-primitive and flag-transitive almost simple automorphism group $G$ with socle $X$ being a finite simple exceptional group of Lie type, and let $H=G_{\alpha}$, for a point $\alpha$ in $\Dmc$. Then $H$ is either parabolic, or one of the subgroups listed in {\rm Tables~\ref{tbl:large-np}-\ref{tbl:num}}.
\end{corollary}

\begin{table}[!htbp]
    \centering
    \scriptsize
    \caption{\small Some large maximal subgroups $H$ of almost simple groups $G$ with socle $X$ a finite simple exceptional group of Lie type.}\label{tbl:large-np}
    \begin{tabular}{lllll}
        \hline
        $X$ &
        $H\cap X$ &
        $l_{v}$ &
        $u_{r}$ &
        Conditions
        \\
        \hline
        ${}^2\!B_2(q^{3})$ &
        ${}^2\!B_2(q)$ &
        $q^{10}$ &
        $q^{4}$ &
        \\
        ${}^2\!G_2(q)$ &
        $A_{1}(q)$ &
        $q^2(q^2-q+1)$&
        $a(q-1)$& \\
        ${}^2\!G_2(q^{3})$&
        ${}^2\!G_2(q)$&
        $q^{14}$ &
        $q^{6}$\\
        ${}^3\!D_4(q)$ &
        $(q^2+\e q+1)A_{2}^{\e}(q)$ &
        $\frac{q^{9}(q^{8}+q^{4}+1)(q^{3}+\e)}{2\cdot (q^{2}+\e q+1)}$ &
        $18a(q-\e)^2(q^2+\e q+1)$ &
        $3\nmid q^2+\e q+1$ \\
        ${}^3\!D_4(q)$ &
        $(q^2+\e q+1)A_{2}^{\e}(q)$ &
        $\frac{q^{9}(q^{8}+q^{4}+1)(q^{3}+\e)}{6\cdot (q^{2}+\e q+1)}$ &
        $48a(q^2+\e q+1)$&
        $3\mid q^2+\e q+1$\\
        ${}^3\!D_4(q)$ &
        $A_{1}(q^3)A_{1}(q)$ &
        $q^8(q^8+q^4+1)$&
        $12a$&
        \\
        ${}^3\!D_4(q)$ &
        $G_2(q)$ &
        $q^6(q^8+q^4+1)$ &
        $12a$ &
        \\
        ${}^3\!D_4(q^{2})$&
        ${}^3\!D_4(q)$ &
        $q^{28}$&
        $486q(q^4+q^2+1)^{2}$&
        $q\geq 4$
        \\
        ${}^2\!F_4(q)$&
        ${}^2B_{2}(q)\wr2$ &
        $q^{14}$ &
        $8a$& \\
        ${}^2\!F_4(q)$ &
        $ B_{2}(q){:}2$ &
        $q^{15}$ &
        $a(q^2-1)^2/2$&
        \\
        ${}^2\!F_4(q^{3})$&
        $ {}^2\!F_4(q)$ &
        $q^{52}$ &
        $q^{18}$ &
        \\
        $G_2(q)$ &
        $A_{2}^{\e}(q)$ &
        $q^{3}(q^{3}+\e1)/2$ &
        $q^{3}-\e1$ &
        \\
        $G_2(q)$&
        ${}^2\!G_2(q)$ &
        $q^3(q+1)(q^3-1)$&
        $2a(q^{2}-q+1)$&
        \\
        $G_2(q)$ &
        $A_{1}(q)^{2}$ &
        $q^4(q^4+q^2+1)$&
        $32a$&
        \\
        $G_2(q^{2})$&
        $G_2(q)$&
        $q^{14}$ &
        $324q(4q^2-1)$ &
        $q\geq 7$
        \\
        $G_2(q^{3})$&
        $G_2(q)$&
        $q^{28}$ &
        $3q^{3}$ &
        \\
        $F_4(q)$ &
        $B_4(q)$ &
        $q^{32}$ &
        $2a(q^{4}+1)$ &
        \\
        $F_4(q)$&
        $D_4(q)$ &
        $q^{24}/6$ &
        $24(q^{4} -1)$&
        \\
        $F_4(q)$&
        $A_{1}(q)C_3(q)$ &
        $q^{28}$ &
        $a(q^{2}-1)^4$&
        \\
        $F_4(q)$&
        $C_{4}(q)$ &
        $q^{22}$ &
        $2a(q^{4}+1)$ &
        \\
        $F_4(q)$&
        $C_{2}(q^2)$ &
        $q^{28}$ &
        $100a(q^{4}+1)$& \\
        $F_4(q)$&  ${C_{2}(q)}^2$ &
        $q^{29}$ &
        $4a(q^4-1)^2(q^2-1)^2$ & \\
        $F_4(q)$&
        ${}^2\!F_4(q)$ &
        $q^{25}$ &
        $2a(q^{2}+1)^{2}(q^{4}-q^{2}+1)$ &
        \\
        $F_4(q)$&
        ${}^3\!D_4(q)$ &
        $q^{22}/3$ &
        $3(q^{8}+q^{4}+1)$ &
        \\
        $F_4(q)$&
        $A_{1}(q)G_{2}(q)$ &
        $q^{34}/4a^2$ &
        $4a^2q^{10}$ &
        \\
        $F_4(q^{2})$&
        $F_4(q)$ &
        $q^{52}$&
        $4q^{17}$&
        \\
        $F_4(q^{3})$&
        $F_4(q)$ &
        $q^{104}$&
        $6q^{17}$&
        \\
        $E_6^{\e}(q)$ &
        $A_{1}(q)A_{5}^{\e}(q)$ &
        $q^{39}/3$ &
        $27q^{11}$ &

        \\
        $E_6^{\e}(q)$&
        $F_4(q)$ &
        $q^{24}/6a$ &
        $q^{8}+q^{4}+1$& \\
        $E_6^{\e}(q)$&
        $(q-\e1)D_5^{\e}(q)$ &
        $q^{29}/3$ &
        $20(q-\e1)^{2}(q^{4}+1)$& \\
        $E_6^{\e}(q)$&  $C_4(q)$ &
        $q^{39}/3$ &
        $3q^4$ &

        \\
        $E_6^{\e}(q)$&
        $(q^2+\e q+1){\cdot}{}^3\!D_4(q)$ &
        $q^{43}/9$ &
        $64eaq^{12}(q^2+\e q+1)$&
        $(\e,q){\neq}(-,2)$
        \\
        $E_6^{\e}(q)$&
        $(q-\e1)^2{{\cdot}}D_4(q)$ &
        $q^{42}/6$ &
        $48aq^{4}(q-\e1)^{6}(q+\e1)^{4}$ &
        $(\e,q){\neq} (+,2)$, $e=1$
        \\
        $E_6^{\e}(q)$&
        $(q-\e1)^2{{\cdot}}D_4(q)$ &
        $q^{44}/288$ &
        $2^{8}\cdot 3^2(q-\e1)^{6}$&
        $(\e,q){\neq} (+,2)$, $e=3$
        \\
        $E_6^{\e}(q^2)$&
        $E_6^{\e'}(q)$&
        $q^{78}$ &
        $18q^{25}$&
        $\e=+$
        \\
        $E_6^{\e}(q^3)$&
        $E_6^{\e}(q)$ &
        $q^{156}$&
        $18q^{43}$&
        \\
        $E_7(q)$ &
        $(q-\e1)E_{6}^{\e}(q)$ &
        $q^{48}$ &
        $3(q^{9}-\e 1)$&
        \\
        $E_7(q)$&
        $A_{1}(q)D_6(q)$ &
        $q^{49}$ &
        $12(q^{2}-1)(q^{8}-1)$&
        \\
        $E_7(q)$&
        $A_{7}^{\e}(q)$
        &
        $q^{63}$ &
        $3(q-\e1)^{2}(q+\e2)$&

        \\
        $E_7(q)$&
        $A_{1}(q)F_4(q)$ &
        $q^{53}/8a^2$ &
        $2^{12}a^8q^{4}(q^2+1)^2(q^4+1)$ &
        \\
        $E_7(q^{2})$ &
        $E_7(q)$  &
        $q^{133}$ &
        $4q^{49}$ &
        \\
        $E_7(q^{3})$
        &  $E_7(q)$  &
        $q^{266}$ &
        $6q^{43}$ &
        \\
        $E_8(q)$ &
        $A_{1}(q)E_7(q)$ &
        $q^{109}$ &
        $4q^{16}$ &
        \\
        $E_8(q)$ &
        $D_{8}(q)$ &
        $q^{124}$ &
        $aq^{52}$ &
        \\
        $E_8(q)$ &
        $A_{2}^{\e}(q)E_6^{\e}(q)$ &
        $q^{156}$ &
        $2aq^{47}$ &
        \\
        $E_8(q^{2})$ &
        $E_8(q)$ &
        $q^{248}$ &
        $12q^{89}$ &
        \\
        $E_8(q^{3})$ &
        $E_8(q)$ &
        $q^{496}$ &
        $18q^{69}$ &
        \\
        \hline
        Note: & \multicolumn{4}{p{11cm}}{\tiny The value $l_v$ in the third column is a lower bound for parameter $v$. The value $u_r$ in the fourth column is an upper bound for the parameter $r$. Here $\e=\pm$, $\e'=\pm$ and $e=\gcd(3,q-\e1)$.}\\
    \end{tabular}
\end{table}

\begin{table}[h]
    \centering
    \scriptsize
    \caption{\small Some large maximal subgroups $H$ of almost simple groups $G$ with socle $X$ a finite simple exceptional group of Lie for small $q$.}\label{tbl:num}
    \begin{tabular}{llll}
        \hline
        $X$ &
        $H\cap X$ &
        $v$ &
        $r$ divides \\
        \hline
        ${}^2\!B_2(8)$ & $13{:}4$ & $560$ &  $156$ \\
        ${}^2\!B_2(32)$ & $41{:}4$ & $198400$ & $820$ \\
        ${}^2\!G_2(4)$ & ${}^2\!G_2(2)$ & $20800$& $12096$ \\
        ${}^2\!G_2(9)$ & ${}^2\!G_2(3)$ & $5321700$& $4245696$ \\
        ${}^2\!G_2(16)$ & ${}^2\!G_2(4)$ & $285282304$& $503193600$ \\
        ${}^2\!G_2(25)$ & ${}^2\!G_2(5)$ & $6348062500$& $5859000000$ \\
        ${}^3\!D_4(2)$ & $7^2{:}SL_{2}(3)$ & $179712$ & $3528$ \\
        ${}^3\!D_4(4)$ & ${}^3\!D_4(4)$ & $320819200$ & $634023936$ \\
        ${}^3\!D_4(9)$ & ${}^3\!D_4(3)$ & $25143164583300$ & $61682494700736$ \\
        ${}^2\!F_4(8)$ & $SU_{3}(8){:}2$ & $8004475184742400$ & $99283968$ \\
        ${}^2\!F_4(8)$& $PGU_{3}(8){:}2$ & $8004475184742400$ & $99283968$  \\
        $G_2(3)$ & $2^{3}{\cdot}A_{2}(2)$ & $3528$ & $192$ \\
        $G_2(4)$ & $A_{1}(13)$ & $230400$ & $2184$\\
        $G_2(4)$& $J_{2}$ & $416$ & $1209600$ \\
        $G_2(5)$ & $G_2(2)$ & $484375$ & $12096$ \\
        $G_2(5)$& $2^{3}{\cdot}A_{2}(2)$ & $4359375$ & $1344$\\
        $G_2(7)$ & $G_2(2)$ & $54925276$ & $12096$ \\
        $G_2(11)$ & $J_{1}$ & $2145199320$ & $175560$\\
        $F_4(2)$  & ${}^3\!D_4(2)$ & $15667200$ & $422682624$\\
        $F_4(2)$& $D_{4}(2)$ & $3168256$ & $1045094400$ \\
        $F_4(2)$& $Alt_{9}$ & $18249154560$ & $362880$ \\
        $F_4(2)$& $Alt_{10}$ & $1824915456$& $3628800$\\
        $F_4(2)$& $A_3(3){\cdot} 2$ & $272957440$ & $24261120$ \\
        $F_4(2)$& $J_{2}$ & $2737373184$ & $1209600$ \\
        $F_4(2)$& $(Sym_6 \wr Sym_2){\cdot}2$ & $3193602048$ & $2073600$ \\
        $E_7(2)$ & $Fi_{22}$ & $123873281581429293827751936$ & $64561751654400$ \\
        $E_7(2)$& $A_{7}^{+}(2)$ & $373849134340755161088$ & $21392255076846796800$\\
        $E_7(2)$& $A_{7}^{-}(2)$ & $268914162119825424384$ &$29739884203317657600$  \\
        $E_7(2)$& $E_{6}^{-}(2)$ & $2488042946297856$& $3214364146718543865446400$\\
        $E_6^{-}(2)$ & $J_{3}$ &$253925177425920$ & $301397760$ \\
        $E_6^{-}(2)$ & $Alt_{12}$ & $319549996007424$ & $1437004800$ \\
        $E_6^{-}(2)$ & $B_{3}(3){:}2$ & $16690645565440$ & $55024220160$ \\
        $E_6^{-}(2)$ & $Fi_{22}$ & $1185415168$ & $387370509926400$\\
        $E_6^{-}(2)$ & $D_{5}^{-}(2)$ & $1019805696$ & $75046138675200$\\
        \hline
    \end{tabular}
\end{table}

\section{Proof of the main result}\label{sec:proof}

Suppose that $\Dmc$ is a nontrivial $(v, k, \lambda)$ design admitting a flag-transitive almost simple automorphism group $G$ with socle $X$ being a finite simple exceptional group of Lie type. Lemma~\ref{lem:New} yields
\begin{align}
v=\frac{|X|}{|H\cap X|}.\label{eq:v}
\end{align}

If the replication number $r$ is coprime to $\lambda$, then Proposition \ref{prop:flag} implies that $G$ is point-primitive, or equivalently, the point-stabiliser $H=G_{\alpha}$ is maximal in $G$, for some point $\alpha$ of $\Dmc$. We now apply Corollary~\ref{cor:large} and conclude that $H$ is either parabolic, or one of the subgroups listed in Tables~\ref{tbl:large-np} and \ref{tbl:num}. We now run through these possible subgroups and prove the main result.

We first assume that $H$ is not parabolic. If $X$ and $H$ are as in one of the rows of Table \ref{tbl:large-np}, then one can find the value of the parameter $v$ and the explicit structure of $H$ in \cite{a:ABD-Exp} and therein references. We now easily obtain a lower bound $l_v$ for $v$ and an upper bound $u_{r}$ for $r$ as in the third and fourth columns of Table \ref{tbl:large-np}, respectively. For each case recorded in Table \ref{tbl:large-np}, we observe that $u_{r}^2<l_{v}$, and this  implies that $r^2<v$ which contradicts Lemma \ref{lem:six}(c). If $X$ and $H\cap X$  are as in Table~\ref{tbl:num}, then by \eqref{eq:v} and Lemma~\ref{lem:six}(c), the parameters  $v$ and $r$ are as in the third and fourth columns of Table~\ref{tbl:num}, respectively. For each value of $v$, we also know that $r$ divides $v-1$, and so for each pair $(v,r)$, the parameter $b$ must divide $vr$, and then the parameter $k$ can be obtained by $k=vr/b$, and finally we can find $\lambda=r(k-1)/(v-1)$. We note also that these parameters must satisfy $\lambda v<r^2$. We then observe that the list of subgroups in Table \ref{tbl:num} gives rise to no possible parameters.

Therefore, $H$ is a parabolic subgroup of $G$. In what follows, we further assume that $K=G_B$ and $K_{0}=K\cap X$, where $B$ is a block containing $\alpha$. We now continue our argument by case by case analysis. We note here that the value of the  parameter $v$ in each case, can be read off from \cite[Table 4]{a:ABD-Exp}.

Suppose first that $X={}^2\!B_{2}(q)$ and $H\cap X=q^{2}{:}(q-1)$ with $q=p^a=2^{2m+1}\geq 8$. We may assume by \cite{b:K-Ex-linear} that $\lambda\geq 2$. Then by \eqref{eq:v}, we have that  $v=q^2+1$, and so $|v-1|_p=q^2$. Note by Lemma~\ref{lem:subdeg} that $G$ has a subdegree $p^{c}$. Thus Lemma \ref{lem:six}(a) and (d) implies that $r$ divides $\gcd(v-1,p^s)$, and hence $r=p^{t}$ is a divisor of $|v-1|_p=q^2$, for some positive integer $t\leq 2a$. Then Lemma \ref{lem:six}(a) yields $k=p^{2a-t}\lambda+1$. Note that $b=rv/k$, $v=q^2+1$ and $G$ is transitive on the set of blocks of $\Dmc$. Then $|G:K|=b=p^t(q^2+1)/k$, where $K=G_{B}$ with $\alpha \in B$. Assume now that  $M_{0}$ is a maximal subgroup of $X$ containing $K_{0}=K\cap X$. Then $|X:M_{0}|$ must divide $b$. The knowledge of maximal subgroups of $X={}^2\!B_{2}(q)$ shows that $K_{0}$ embeds into a parabolic subgroup $M_{0}\cong q^{2}{:}(q-1)$ of index $q^{2}+1$. Again, we must have $|X:K_{0}|$ dividing $b=p^t(q^2+1)/k$, and since $q^{2}+1$ divides $|X:K_{0}|$, it follows that $k$ divides $p^{t}$.  Since also $k=p^{2a-t}\lambda+1$, we conclude that $t=2a$ and $k=\lambda+1=p^{n}$ for some $n\leq 2a$, and hence $r=q^{2}$ and $b=p^{2a-n}(q^{2}+1)$.
Moreover, \cite[Corollary 1.3]{a:ABD-Exp} implies that $n\neq 2a$, and hence $K_{0}$ is properly contained in $M_{0}$. From now on we can assume that $M_{0}=(H_{0})_{\beta}$, for some $\beta \in \Pmc$. Moreover, as $X$ is $2$-transitive on points, \cite[2.2.8]{b:dembowski}, the  group $X$ is also flag-transitive. We no apply \cite[Lemma 2]{a:Downs-Suz-16} and conclude that $X_{\alpha,\beta}=H_{0}\cap M_{0}$ is isomorphic to a cyclic group of order $q-1$ which has index $q^{2}$ in $H_{0}$. But $H_{0}\cap K_{0}=X_{\alpha,B}$ is contained in $X_{\alpha,\beta}$ and $|H_{0}:H_{0}\cap K_{0}|=r=q^{3}$. Thus $X_{\alpha,B}=H_{0}\cap K_{0}=X_{\alpha,\beta}$, hence all subgroups of $X_{\alpha,\beta}$ fix the block $B$. Since the non-trivial $X_{\alpha,\beta}$-orbits are of length $q-1$ or $(q-1)/2$, it follows that $B\setminus\{\alpha\}$ is union of non-trivial $X_{\alpha,\beta}$-orbits, that is to say, $(q-1)/2$ divides $k-1$. Recall that $k=p^{n}$, for some $1\leq n<2n$. Therefore, $k=q$, as desired in part (a).

Suppose now that $X={}^2\!G_{2}(q)$ and $H\cap X=q^3{:}(q-1)$ with $q=p^{a}=3^{2m+1}\geq 9$. \eqref{eq:v}, we have that  $v=q^3+1$.  Lemma~\ref{lem:subdeg} implies that $G$ has a subdegree $p^{s}$. Thus Lemma \ref{lem:six}(a) and (d) implies that $r$ divides $\gcd(v-1,p^s)$, and hence $r=p^{t}$ is a divisor of $|v-1|_p=q^3$. Since also $v<r^{2}$, it follows that $3a/2<t\leq 3a$. Lemma \ref{lem:six}(a) yields $k=p^{3a-t}\lambda+1$. Since $b=rv/k$ and $v=q^3+1$,  we have that $|G:K|=b=p^t(q^3+1)/k$, where $K=G_{B}$ with $\alpha\in B$.
Let $M$ be a maximal subgroup of $G$ containing $K$ and $M_{0}=K\cap M$. Then $M_{0}$ is a maximal subgroup of $X$ containing $K_{0}=K\cap X$. Then $|G:M|=|X:M_{0}|$ is a divisor of $b$, and so by running through the list of maximal subgroups of $X=^{2}\!G_{2}(q)$ in \cite[Table 8.43]{b:BHR-Max-Low}, we find that the only possibilities are $2\times A_{1}(q)$ and $q^3{:}(q-1)$. If $|G:X|=c$ with $c$ a positive integer dividing $|\Out(X)|=a$, then $M$ is isomorphic to either $(2\times A_{1}(q))\cdot c$, or $(q^3{:}(q-1))\cdot c$.

Let first $M_{0}$ be isomorphic to $2\times A_{1}(q)$. Then $q^{2}(q^{2}-q+1)$ divides $b=p^{t}(q^{3}+1)/k$, and so $k$ divides $p^{t-2a}(q+1)$ implying also that $2a\leq t$.
Note also that $|X:K_{0}|$ divides $b$. Then by inspecting the maximal subgroups of $A_{1}(q)$ containing $K_{0}$, we easily observe that $K_{0}$ is isomorphic to either $A_{1}(q)$, or $2\times A_{1}(q)$. Since the point stabiliser $H$ is a solvable group, the subgroup $A_{1}(q)$ of $K_{0}$ acts non-trivially on the block $B$, and since $q+1$ is the smallest degree of non-trivial action of $A_{1}(q)$, we have that $k\geq q+1$, or equivalently, $\lambda \geq p^{t-2a}$ with $2a\leq t\leq 3a$.

If $k=p^{3a-t}\lambda +1$ is coprime to $p$, then since $k\mid p^{t-2a}(q+1)$, it follows that $k$ divides $q+1$, that is to say, $k\leq q+1$, or equivalently, $\lambda\leq p^{t-2a}$. We have already proved that $k\geq q+1$ and $\lambda \geq p^{t-2a}$. Therefore, $k=q+1$, $\lambda=p^{t-2a}$, $r=p^{t}$ and $b=p^{t}(q^2-q+1)$ with $t\geq 2a$. The assumption that $r$ is coprime to $\lambda$ forces $t= 2a$, and so $\lambda=1$. We now apply \cite[Theorem A]{b:K-Ex-linear} and conclude that $\Dmc$ is the Ree Unital space with $b=q^{2}(q^{2}-q+1)$, $r=q^{2}$ and $k=q+1$. This is part (b), as claimed.

If $k=p^{3a-t}\lambda +1$ is divisible by $p$, then $k=\lambda+1$ and $t=3a$. This implies that $r=q^3$ and $k$ divides $p^{t-2a}(q+1)=q(q+1)$. Since $b=q^{3}(q^{3}+1)/k$ and $|G:M|=q^{2}(q^{2}-q+1)$, it follows that $|M:K|=q(q+1)$ and $|K|=ck(q-1)$ where $c=|G:X|$. This also implies that $G=KX$. Recall that $K_{0}$ is isomorphic to $A_{1}(q)$ or $2\times A_{1}(q)$, and so $|M:K_{0}|$ is $m$ or $2m$, respectively. On the other hand,  $|M:K_{0}|=|M:K|\cdot |K:K_{0}|=[q(q+1)/k]\cdot m$. Therefore, $k$ is $q(q+1)$ or $q(q+1)/2$ if $K_0$ is isomorphic to $2\times A_{1}(q)$ or $A_{1}(q)$, respectively. Since $b=p^{3}(q^{3}+1)/k$ and $\lambda=k-1$, if $k=q(q+1)$, then $b=q^{2}(q^{2}-q+1)$ and $\lambda=q^{2}+q-1$, and if $k=q(q+1)/2$, then $b=2q^{2}(q^{2}-q+1)$ and $\lambda=(q^{2}+q-2)/2$. This means in both cases that $q$ is not a divisor of $|H\cap K|$, and this is impossible by \cite[Lemma 3.2]{a:Pierro-Ree-16}.

Let now $M_{0} \cong q^{3}{:}(q-1)$. Since $H_{0}$ is the parabolic subgroup isomorphic to $q^{3}{:}(q-1)$, we may can assume that $M_{0}=X_{\beta}$, for some $\beta \in \Pmc$. Moreover, $X$ is $2$-transitive, and so by \cite[2.2.8]{b:dembowski}, the group $X$ is also flag-transitive. Recall in this case that $r=p^{t}$ is a divisor of $|v-1|_p=q^3$ and $k=p^{3a-t}\lambda+1$, for some positive integer $t\leq 3a$. As $|G:M|=|X:M_{0}|=q^{3}+1$ divides $b=p^{t}(q^{3}+1)/k$, we conclude that $k$ divides $p^{t}$.  Since also $k=p^{3a-t}\lambda+1$, we conclude that $t=3a$ and $k=\lambda+1=p^{n}$ for some $n\leq 3a$, and hence $r=q^{3}$ and $b=p^{3a-n}(q^{3}+1)$.
Note by \cite[Corollary 1.3]{a:ABD-Exp} that $n\neq 3a$, and hence $K_{0}$ is properly contained in $M_{0}$.
We now apply \cite[Lemma 3.2]{a:Pierro-Ree-16}, and conclude that $X_{\alpha,\beta}=H_{0}\cap M_{0}$ is isomorphic to a cyclic group of order $q-1$ which has index $q^{3}$ in $H_{0}$, and since $H_{0}\cap K_{0}=X_{\alpha,B}$ is contained in $X_{\alpha,\beta}$ and $|H_{0}:H_{0}\cap K_{0}|=r=q^{3}$, we conclude that $X_{\alpha,B}=H_{0}\cap K_{0}=X_{\alpha,\beta}$. This in particular says that all subgroups of $X_{\alpha,\beta}$ fix the block $B$ including $L:=X_{\alpha,\beta,\gamma}$, for all $\gamma\in B\setminus\{\beta\} $. We know that $L$ is generated by an involution, and hence the non-trivial $X_{\alpha,\beta}$-orbits are of length $q-1$ or $(q-1)/2$, respectively, when $\gamma$ is fixed by $X_{\alpha,\beta}$ or not. Since $B$ is fixed by $X_{\alpha,\beta}$, it follows that $B\setminus\{\alpha\}$ is union of non-trivial $X_{\alpha,\beta}$-orbits, that is to say, $(q-1)/2$ divides $k-1$. Recall that $k=p^{n}$, for some $1\leq n<3n$. Therefore, $k$ is $q$ or $q^{2}$, and hence $\lambda$ is $q-1$ or $q^{2}-1$, respectively, and this follows parts (c) and (d).

Suppose that $X=E_{6}(q)$ and
$H\cap X=[q^{16}]{:}D_{5}(q)\cdot (q-1)$. Then $v=(q^8+q^4+1)(q^9-1)/(q-1)$. Note by \cite{a:LSS-rank3} that $G$ has nontrivial subdegrees  $q(q^{8}-1)(q^{3}+1)/(q-1)$ and $q^{8}(q^{5}-1)(q^{4}+1)/(q-1)$, and so by Lemma~\ref{lem:six}(d), we conclude that $r$ divides $q(q^{4}+1)$, and so Lemma \ref{lem:six}(c) implies that $v<r^2<q^{12}$, which is a contradiction.

Suppose that $X=E_{6}(q)$ and $H\cap X=[q^{25}]{:}A_{1}(q)A_{4}(q){\cdot}(q-1)$. Then $v=(q^{3}+1)(q^{4}+1)(q^{9}-1)(q^{12}-1)/(q-1)(q^{2}-1)$. It follows from \cite{a:Korableva-E6E7} that $X$ has subdegrees $q(q^5-1)(q^4-1)/(q-1)^2$ and $q^{13}(q^{5}-1)(q-1)/(q-1)^{2}$. Moreover, by Lemma \ref{lem:six}(a) and (c), the parameter $r$ is a divisor of $\gcd(v-1,|H|)$. Therefore, Lemma \ref{lem:six}(d) implies that $r$ divides $6aq(q^{4}+q^{3}+q^{2}+q+1)$, and so $r^2<q^{16}<v$, which is impossible by Lemma \ref{lem:six}(c).

We finally consider the remaining cases.  Then by \cite[Table 4]{a:ABD-Exp}, we easily observe that
\begin{align}\label{eq:parab}
|v-1|_p=
\begin{cases}
q^{3}, & \text{ if $X={}^3\!D_{4}(q)$ and $H\cap X=[q^{11}]{:}SL_{2}(q) {\cdot}(q^{3}-1)$; } \\
q^{2}, &\text{ if $X={}^2\!F_{4}(q)$ and $H\cap X=[q^{11}]{:}GL_{2}(q)$; }\\
q, & \text{ otherwise.}
\end{cases}
\end{align}
Note by Lemma~\ref{lem:subdeg} and Remark~\ref{rem:subdeg} that there is a unique subdegree $p^{c}$. Then by Lemma~\ref{lem:six}, $r$ must divide $\gcd( v-1,p^{c})$, and so $r$ divides $|v-1|_p$ as in \eqref{eq:parab}. Thus $v<(|v-1|_p)^2$ by Lemma \ref{lem:six}(c), which is impossible. For example, if $X={}^3\!D_{4}(q)$ and $H\cap X=[q^{11}]{:}SL_{2}(q) {\cdot}(q^{3}-1)$, then $v=(q^8+q^4+1)(q^3+1)$ and $r$ divides $|v-1|_p=q^3$, and so  $q^{11}<v<r^2<q^6$, which is a contradiction.

\subsection*{Acknowledgements}

The author would like to thank anonymous referees for providing us helpful and constructive comments and suggestions. The author are also grateful to Cheryl E. Praeger and Alice Devillers  for supporting his visit to UWA (The University of Western Australia) during July-September 2019. He would also like to thank Alexander Bors for introducing reference \cite{a:Pierro-Ree-16}.



\end{document}